\documentclass[11pt]{amsart}
\usepackage{amssymb,amsmath,epsfig,graphics,mathrsfs}

\usepackage{eucal,url,amssymb,stmaryrd,enumerate,amscd,verbatim}
\usepackage{amsmath}
\usepackage[pagebackref,colorlinks=true,linkcolor=blue,citecolor=blue]{hyperref}
\usepackage{amsfonts}
\usepackage{amsthm}
\usepackage[margin=1in]{geometry}
\usepackage{eurosym,lipsum,latexsym,esint}
\numberwithin{equation}{section}
\newtheorem{thrm}{Theorem}[section]
\newtheorem{lemma}[thrm]{Lemma}
\newtheorem{prop}[thrm]{Proposition}

\newcommand{\COG}{C^\infty_0(\bG)}
\newcommand{\DsG}{\mathcal{D}^{s,2}(\bG)}

\newcommand{\eps}{\varepsilon}

\newcommand{\p}{\partial}
\newcommand{\G}{\Gamma}
\newcommand{\la}{\lambda}
\newcommand{\ve}{\varepsilon}

\newcommand{\R}{\mathbb{R}}
\newcommand{\Rn}{\mathbb{R}^n}

\newcommand{\Om}{\Omega}

\newcommand{\Hn}{\mathbb{H}^n}

\newcommand{\bG}{\mathbb{G}}
\newcommand{\algg}{\mathfrak g}
\newcommand{\frlap}{\mathcal{L}_s}%{(-\triangle)^s}%\mathscr{D}_s

\newcommand{\norm}[1]{\lVert#1\rVert}

\newcommand{\volg}{\, d{g}}
\newcommand{\vol}{\, d{h}}

\newcommand*\MSC[1][1991]{\par\leavevmode\hbox{%
\textit{\,\,\,\,\, #1 Mathematical subject classification:\ }}}
\newcommand\blfootnote[1]{%
  \begingroup
  \renewcommand\thefootnote{}\footnote{#1}%
  \addtocounter{footnote}{-1}%
  \endgroup
  }

\begin{document}

\title[Optimal decay for solutions, etc.]{Optimal decay for solutions of nonlocal semilinear  equations with critical exponent in homogeneous groups}

{\blfootnote{\MSC[2020]{35R03, 35R11, 53C18}}}


\thanks{N.G. was supported in part by a SID Project: ``Nonlocal Sobolev and isoperimetric inequalities", Univ. of Padova, 2019, and by a BIRD grant: ``Aspects of nonlocal operators via fine properties of heat kernels", Univ. of Padova, 2022. D.V. was partially supported by the Advanced Research Projects Agency-Energy (ARPA-E), U.S. Department of Energy, under Award Number DE-AR0001202. The views and opinions of authors expressed herein do not necessarily state or reflect those of the United States Government or any agency thereof.}

%\date{\today}

\begin{abstract}
In this paper we establish the sharp asymptotic decay of positive solutions of the Yamabe type equation $\frlap u=u^{\frac{Q+2s}{Q-2s}}$ in a homogeneous Lie group, where $\frlap$ represents a suitable pseudodifferential operator modelled on a class of nonlocal operators arising in conformal CR geometry.
\end{abstract}

\author[]{Nicola Garofalo}
\author[]{Annunziata Loiudice }
\author[]{Dimiter Vassilev }

\maketitle

%\vsmall
\section{Introduction}\label{S:intro}

The study of the CR Yamabe problem began with the celebrated works of Jerison and Lee \cite{JL1}-\cite{JL4}. The prototypical nonlinear PDE in this problem is
\[
\mathscr L u = u^{\frac{Q+2}{Q-2}},
\]
where $\mathscr L$ indicates the negative sum of squares of the  left-invariant vector fields generating the horizontal space   in the Heisenberg group $\Hn$ with real dimension $2n+1$, whereas $Q = 2n+2$ denotes the so-called homogeneous dimension associated with the non-isotropic group dilations\footnote{In this paper we always use the group law dictated by the Baker-Campbell-Hausdorff formula. When the Lie group is $\Hn$, or more in general a group of Heisenberg type, this choice obviously affects the expression of the horizontal Laplacian.}. In the present paper we are interested in the following nonlocal version of the above equation
\begin{equation}\label{main}
\mathscr L_s u = u^{\frac{Q+2s}{Q-2s}},
\end{equation}
where the fractional parameter $s\in (0,1)$, and $\mathscr L_s$ denotes a certain pseudodifferential operator which arises in conformal CR geometry. As an application of our main result we derive sharp decay estimates of nonnegative solutions of \eqref{main}.

The operator $\mathscr L_s$ in \eqref{main} was first introduced in \cite{BFM} via the spectral formula
\begin{equation}\label{bfm}
\mathscr L_s = 2^s |T|^s \frac{\G(-\frac 12\mathscr L |T|^{-1} + \frac{1+s}2)}{\G(-\frac 12 \mathscr L |T|^{-1} + \frac{1-s}2)},
\end{equation}
where $\G(x) = \int_0^\infty t^{x-1} e^{-t} dt$ denotes Euler gamma function. In \eqref{bfm} we have let $T = \p_\sigma$, where for a point $g\in \Hn$ we have indicated with $g = (z,\sigma)$ its logarithmic coordinates.
More in general, in a group of Heisenberg type $\bG$, with logarithmic coordinates $g= (z,\sigma)\in \bG$, where $z$ denotes the horizontal variable and $\sigma$ the vertical one, the pseudodifferential operator $\mathscr L_s$ is defined by the following generalisation of  \eqref{bfm}
\begin{equation}\label{bfmH}
\mathscr L_s = 2^s (-\Delta_\sigma)^{s/2} \frac{\G(-\frac 12\mathscr L (-\Delta_\sigma)^{-1/2} + \frac{1+s}2)}{\G(-\frac 12 \mathscr L (-\Delta_\sigma)^{-1/2} + \frac{1-s}2)},
\end{equation}
where $-\Delta_\sigma$ is the positive Laplacian in the center of the group, see \cite{RT}. Formulas \eqref{bfm}, \eqref{bfmH} should be seen as the counterpart of the well-known spectral representation $\widehat{(-\Delta)^s u} = (2\pi |\xi|)^{2s} \hat u$, where we have denoted by $\hat f$ the Fourier transform of a function $f$, see \cite[Chap. 5]{St}. An important fact, first proved for $\Hn$ in \cite{FGMT} using hyperbolic scattering, and subsequently generalised to any group of Heisenberg type in \cite{RT} using non-commutative harmonic analysis, is the following Dirichlet-to-Neumann characterisation of $\mathscr L_s$
\[
- \underset{y\to 0^+}{\lim} y^{1-2s} \p_y U((z,\sigma),y) = 2^{1-2s} \frac{\G(1-s)}{\G(s)}\mathscr L_s u(z,\sigma),
\]
where $U((z,\sigma),y)$ is the solution to a certain extension problem from conformal CR geometry very different from that of Caffarelli-Silvestre in \cite{CS}. Yet another fundamental fact, proved in
\cite[Proposition 4.1]{RTaim} and \cite[Theorem 1.2]{RT} for $0<s<1/2$, is the following remarkable M. Riesz type representation
\begin{equation}\label{riesz}
 \alpha(m,k,s)\  \mathscr L_s u(g)=\int_{\bG} \frac {u(g)-u(h)}{|h|^{Q+2s}}dh,
\end{equation}
where with $g = (z,\sigma)$, we have denoted by $|g| = |(z,\sigma)| = (|z|^4 + 16 |\sigma|^2)^{1/4}$ the non-isotropic gauge in a group of Heisenberg type $\bG$. Using the heat equation approach in \cite{GTaim}, \cite{GTinter}, formula \eqref{riesz} can be extended to cover the whole range $0<s<1$. In \eqref{riesz} the number $\alpha(m,k,s)>0$ denotes an explicit constant depending on $s$ and the dimensions $m$ and $k$ of the horizontal and vertical layers of the Lie algebra of $\bG$.
While by \eqref{bfm}, \eqref{bfmH} and the classical formula $\G(x+1) = x \G(x)$, it is formally almost obvious that in the limit as $s\nearrow 1$ the operator $\mathscr L_s$ tends to the horizontal Laplacian $\mathscr L$, we emphasise that, contrarily to an unfortunate misconception, when $\bG = \Hn$, or more in general it is of Heisenberg type, for no $s\in (0,1)$ does the standard fractional power
\begin{equation}\label{bala}
\mathscr L^s u(g) \overset{def}{=}  (-\mathscr L)^s u(g) = - \frac{s}{\G(1-s)} \int_0^\infty \frac{1}{t^{1+s}} \left(P_t u(g) - u(g)\right) dt
\end{equation}
coincide with the pseudodifferential operator defined by the left-hand side of \eqref{riesz}\footnote{in \eqref{bala} we have denoted by $P_t = e^{-t\mathscr L}$ the heat semigroup constructed in \cite{Fo}.}. Unlike their classical predecessors $(-\Delta)^s$, in the purely non-Abelian setting of $\Hn$ the pseudodifferential operators $\mathscr L^s$ in \eqref{bala} are \emph{not} CR conformally invariant, nor they have any special geometric meaning, while the operators  $\mathscr L_s$ are CR conformally invariant.  For these reasons,  we will refer to  the operator $\mathscr L_s$ as the \emph{geometric (or conformal)  fractional sub-Laplacian} even in the general setting of group of Heisenberg type, see \cite[Section 8.3]{FGMT} for relevant remarks in  the remaining non-Abelian groups of Iwasawa type.  Furthermore, it is not true that the fundamental solution $\mathscr E^{(s)}(z,\sigma)$ of $\mathscr L^s$ is a multiple of $|(z,\sigma)|^{2s-Q}$, see \cite[Theor. 5.1]{GTaim}. What is instead true, as proven originally by Cowling and Haagerup  \cite{CH89}, see also \cite[(3.10)]{RTaim},  and with a completely different approach based on heat equation techniques in  \cite[Theor. 1.2]{GTaim} (the reader should also see in this respect the works \cite{GTinter} and \cite{GTpotan}), is that the fundamental solution of the conformal  fractional sub-Laplacian $\mathscr L_s$ in \eqref{bfmH}  is given by
\begin{equation}\label{Esottos}
\mathscr E_{(s)}(z,\sigma) = \frac{C_{(s)}(m,k)}{|(z,\sigma)|^{Q-2s}},
\end{equation}
where
\[
C_{(s)}(m,k) = \frac{2^{\frac m2 + 2k-3s-1} \G(\frac 12(\frac m2+1-s)) \G(\frac 12(\frac m2 + k -s))}{\pi^{\frac{m+k+1}2} \G(s)}.
\]
It is worth emphasising here that, when $s\to 1$, one recovers from \eqref{Esottos} the famous formula for the fundamental solution of $-\mathscr L$, first found by Folland in \cite{Fobams} in $\Hn$, and subsequently generalised by Kaplan in \cite{Ka} to groups of Heisenberg type. Before proceeding we pause to notice that from the stochastic completeness and left-invariance of $P_t$ one tautologically obtains from \eqref{bala}
\[
\mathscr L^s u(g) = \frac {1}{2}\int_{\bG} \frac {2u(g)-u(gh)-u(gh^{-1})}{||h||_{(s)}^{Q+2s}}dh,
\]
where for $h\in \bG$ we have defined
\begin{equation}\label{taut}
\frac{1}{||h||_{(s)}^{Q+2s}} \overset{def}{=}  \frac{2s}{\G(1-s)} \int_0^\infty \frac{1}{t^{1+s}} p(h,t) dt,
\end{equation}
with $p(h,t)$ the positive heat kernel of $-\mathscr L$. While in the Abelian case $\bG = \Rn$, with Euclidean norm $|\cdot|$, an elementary explicit calculation in \eqref{taut}, based on the knowledge that $p(x,t) = (4\pi t)^{-\frac n2} e^{-\frac{|x|^2}{4t}}$, gives
\[
\frac{1}{||x||_{(s)}^{n+2s}} = \frac{s 2^{2s+1} \G(\frac{n+2s}2)}{\pi^{\frac n2}\G(1-s)} \frac{1}{|x|^{n+2s}},
\]
and one recovers from \eqref{bala} Riesz' classical representation, when $\bG = \Hn$ it is not true that the right-hand side of \eqref{taut} defines a function of the gauge $|g| = |(z,\sigma)| = (|z|^4 + 16 |\sigma|^2)^{1/4}$. In any group of Heisenberg type a close form explicit expression of the right-hand side of \eqref{taut} was computed in \cite[Theor. 5.1]{GTaim} (to obtain it one should change $s$ into $-s$ in that result),  starting from the following intertwined heat kernel formulas found in \cite{GTpotan}
\begin{equation}\label{beauty}
\mathscr K_{(\pm s)}((z,\sigma),t) = \frac{2^k}{\left(4\pi t\right)^{\frac m2+k}}\int_{\R^k} e^{-\frac it \langle \sigma,\la\rangle} \left(\frac{|\la|}{\sinh |\la|}\right)^{\frac m2+1\mp s}  e^{- \frac{|z|^2}{4t} \frac{|\la|}{\tanh |\la|}}d\la,\ \ \ \ 0<s\le 1.
\end{equation}
We emphasise that, when $s\nearrow 1$, the function $\mathscr K_{(+ s)}((z,\sigma),t)$ converges to the famous heat kernel discovered independently by Hulanicki \cite{Hu} and Gaveau \cite{Gav}.
The above mentioned \cite[Theor. 5.1]{GTaim}, and formula \eqref{riesz}, imply that $\mathscr L^s \not= \mathscr L_s$ for every $0<s<1$.

Formulas \eqref{riesz} and \eqref{Esottos} motivated the results in the present work. As we have mentioned, we are interested in optimal decay estimates for nonnegative subsolutions of \eqref{main}.   In this respect, \cite[Theorem 3.1]{RTaim}  and \cite[Theorem 3.7]{RT} gave the explicit form\label{uy} of a solution to  the fractional Yamabe equation on a group of Heisenberg type  as a consequence of the intertwining properties of $\mathscr L_s$ for $0 <s <n + 1$,  see also \cite{GTinter} for a different approach to intertwining based on the heat equation.  In the notation of \cite[Corollary 3.3]{GTinter},  the result is that if $\bG$ is of Heisenberg type, and $0<s<1$, then for every $(z,\sigma)\in\bG$, and $y>0$
one has the following intertwining identity
\begin{equation}\label{interi}
\mathscr L_s\left(((|z|^2+y^2)^2+16|\sigma|^2)^{-\frac{m + 2k - 2s}4}\right)=\frac{\G\left(\frac{m+2+2s}{4}\right)\G\left(\frac{m+2k+2s}{4}\right)}{\G\left(\frac{m+2-2s}{4}\right)\G\left(\frac{m+2k-2s}{4}\right)} (4y)^{2s}((|z|^2+y^2)^2+16|\sigma|^2)^{-\frac{m + 2k + 2s}4}.
\end{equation}
An immediate consequence of \eqref{interi} is that, for any real positive  number $y>0$, the function
\begin{equation}\label{uy}
u_y(z,\sigma)= \left(\frac{\G\left(\frac{m+2+2s}{4}\right)\G\left(\frac{m+2k+2s}{4}\right)}{\G\left(\frac{m+2-2s}{4}\right)\G\left(\frac{m+2k-2s}{4}\right)}\right)^{\frac{m + 2k - 2s}{4s}}\left(\frac{16y^2}{(|z|^2+y^2)^2+16|\sigma|^2}\right)^{\frac{m + 2k - 2s}4}
\end{equation}
is a positive solution of the nonlinear equation \eqref{main}.
We note that in the particular setting of the Heisenberg group $\mathbb{H}^n$ (which corresponds to the case $m = 2n$ and $k=1$) the function appearing in the left-hand side of \eqref{interi} defines, up to group translations, the unique extremal of the Hardy-Littlewood-Sobolev inequalities obtained by Frank and Lieb in \cite{FL}\footnote{We emphasise that letting $s\nearrow 1$ one recovers from \eqref{uy} the functions that, in the local case $s=1$, were shown to be the unique positive solutions of the CR Yamabe equation respectively in \cite{JL2}, for the Heisenberg group $\Hn$, and \cite{IMV}, for the quaternionic Heisenberg group. See also the important cited work \cite{FL}, and \cite[Theor. 1.1]{GV00} and \cite{GVduke} for partial results in groups of Heisenberg type.}

Whether in a group of Heisenberg type $\bG$ \emph{all} nonnegative solutions of \eqref{main} are, up to left-translations,  given by  \eqref{uy} presently remains a challenging open question. A first step in such problem is understanding the optimal decay of nonnegative solutions to \eqref{main}. Keeping in mind that  the number $m+2k$ in \eqref{uy} represents the homogeneous dimension $Q$ of the group $\bG$, by setting the scaling factor $y=1$, we see that there exists a universal constant $C>0$ such that
\[
u_1(z,\sigma)\le \frac{C}{|(z,\sigma)|^{Q-2s}}.
\]
It is thus natural to guess that the optimal decay of all nonnegative solutions to \eqref{main} should be dictated by \eqref{Esottos}, i.e., by the fundamental solution of $\mathscr L_s$. In Theorem \ref{t:asympt yamabe} below we prove that this guess is correct.

To facilitate the exposition of the ideas and underline the general character of our approach, in this paper we have chosen to work in the setting of homogeneous Lie groups $\bG$ with dilations $\{\delta_\la\}_{\la >0}$, as in the seminal monograph of Folland and Stein \cite{FS82}. We emphasise that such groups encompass the stratified, nilpotent Lie groups in \cite{Eli}, \cite{Fo} and \cite{FS82} (but they are a strictly larger class). In particular, our results include Lie groups of Iwasawa type for which \eqref{main} becomes significant in the case of pseudo-conformal CR and quaternionic contact geometry. We shall assume throughout that $|\cdot|$ is a fixed homogeneous norm in $\bG$, i.e., $g\mapsto |g|$ is a continuous function on $\bG$ which is $C^\infty$ smooth on $\bG\setminus \{e\}$, where $e$ is the group identity, $ |g|=0$ if and only if $g=e$,  and  for all $g\in\bG$ we have
\begin{equation}\label{e:homog norm}
(i) \ |g^{-1}|=|g|; \quad (ii)\ |\delta_\lambda g|=\lambda|g|.
\end{equation}
Finally, we shall assume that the fixed norm satisfies the triangle inequality
\begin{equation}\label{e:triangle ineqs}
 |g\cdot h|\leq |g|+|h|, \qquad g,h\in \bG.
\end{equation}
We stress that, according to \cite{HS90}, any homogeneous group allows a norm which satisfies the triangle inequality\footnote{It is well-known that in a group of Heisenberg type the anisotropic gauge $|g| = |(z,\sigma)| = (|z|^4 + 16 |\sigma|^2)^{1/4}$ satisfies (i)-(iii), see \cite{Cy}.}. We shall denote with
\[
B_R(g)\equiv B(g,R)=\{h\, \mid\, |g^{-1}\circ h|<R\}
\]
the resulting open balls with center $g$ and radius $R$.

Motivated by the above result \eqref{riesz}, for $1\le p < \infty$ and $0<s<1$ we consider the Banach space $\mathcal{D}^{s,p}(\bG)$ defined as the closure of the space of functions $u\in \mathcal{C}^\infty_0 (\bG)$ with respect to the norm
\begin{equation}\label{semi}
||u||_{\mathcal{D}^{s,p}(\bG)}= [u]_{s,p} = \left(\int_{\bG}\int_{\bG}\frac {|u(g)-u(h)|^p}{|g^{-1}\cdot h|^{Q+ps}} dg dh\right)^{1/p} < \infty.
\end{equation}
We are particularly interested in the case $p=2$.  In this  case, the Euler-Lagrange equation of \eqref{semi} involves the following left-invariant nonlocal operator, initially defined on functions $u\in C^\infty_0(\bG)$
\begin{equation}\label{e:frac L via hypersing}
\frlap u(g)=
\frac {1}{2}\int_{\bG} \frac {2u(g)-u(gh)-u(gh^{-1})}{|h|^{Q+2s}}dh
=\lim_{\eps\rightarrow
0}\int_{\bG\setminus B(g,\eps)} \frac {u(g)-u(h)}{|g^{-1}\cdot h|^{Q+2s}}dh,
\end{equation}
 see \cite{GLV22p} for a general construction of the fractional operator $\frlap$  on the Dirichlet space $\mathcal{D}^{s,2}(\bG)$ and relevant Sobolev-type embedding results.
In \eqref{e:frac L via hypersing}, and hereafter in this work, the number $Q>0$ represents the homogeneous dimension of $\bG$ associated with the group dilations $\{\delta_\la\}_{\la >0}$. It is clear from \eqref{riesz} that, when $\bG$ is of Heisenberg type, the nonlocal operator $\frlap$ defined using the Koranyi gauge is just a multiple of $\mathscr L_s$ in \eqref{bfmH}, and this provides strong enough motivation to work with \eqref{e:frac L via hypersing}. A second motivation comes from
\cite[Theor. 1.2]{GLV22p}, in which we prove that, if $X_1,...,X_m$ are the left-invariant vector fields of homogeneity one with associated coordinates $x_j$, and the fixed homogeneous norm $|g|$ is a spherically symmetric function of the coordinates $(x_1,\dots,x_m)$, then for a function $u\in C^\infty_0(\bG)$ we have the identities
\begin{equation}\label{e:limits frac operator}
\lim_{s\rightarrow 0^+} \frac {2s}{\sigma_Q}\frlap u(g)=-u(g) \quad\text{and}\quad \lim_{s\rightarrow 1^-} \frac {4m(1-s)}{\tau_m}\frlap u(g)=-\sum_{i=1}^m X_i^2u(g),
\end{equation}
where $\sigma_Q, \tau_m>0$ are suitable universal constants. Throughout the paper, for  $0<s<1$ we let
\begin{equation}\label{e:2star def}
2^*(s)\overset{def}{=}\frac{2Q}{{Q-2s}} \qquad \text{and}\qquad (2^*(s))'{=}\frac{2Q}{{Q+2s}},
\end{equation}
so that $2^*(s)$ {equals} the Sobolev exponent associated to the fractional Sobolev inequality \cite[Thorem 1.2]{GLV22p}
\begin{equation}\label{e:fracsobolev}
\left ( \int_{\bG} |u|^{2^*(s)}(g){dg}\right)^{1/2^*(s)}\leq S
\left(\int_{\bG}\int_{\bG} \frac {|u(g)-u(h)|^2}{|
h^{-1}\cdot g|^{Q+2s}}\, {dg}dh \right )^{1/2},
\end{equation}
and $(2^*(s))'$ is its H\"older conjugate.
In addition to the fractional Sobolev exponent $2^*(s)$, the following exponents will be used
\begin{equation}\label{e:r def}
 r\overset{def}{=}\frac {2^*(s)}{2}=\frac {Q}{Q-2s}\qquad\text{and}\qquad r'=\frac{r}{r-1} = \frac {Q}{2s}.
\end{equation}

With all this being said, we are ready to state our results. The first one concerns the nonlocal
Schr\"odinger type equation \eqref{e:fracschro} below. For the notion of subsolution to such equation, see \eqref{e:weak formulation} below.

\begin{thrm}\label{t:loc bound} Let $\bG$ be a homogeneous group.
Let $u\in \mathcal{D}^{s,2}(\bG)$ be a nonnegative subsolution to the equation
\begin{equation}\label{e:fracschro}
\frlap u=Vu.
\end{equation}
Suppose the following conditions hold true:
\begin{enumerate}[i)]
\item \label{a:Lp prop of V}for some $t_0>r'= \frac{Q}{2s}$ we have $V\in L^{r'}(\bG)\cap L^{t_0}(\bG)$ ;
\item \label{a:decay of V} there exist $\bar R_0$ and $K_0$ so that for $R>\bar R_0$ we have
\begin{equation}\label{e:decay condn}
\int_{\{|g|>R\}} |V(g)|^{t_0}{dg}\leq \frac {K_0}{R^{2st_0-Q}}.
\end{equation}
\end{enumerate}
Then there exists a constant $C>0$, depending on $Q$, $s$ and $K_0$, such that for all  $g_0\in \bG$ with $|g_0|=2R_0\geq 4\bar R_0$,  we have for $0<R\leq R_0$
\begin{equation}\label{e:loc bound}
\sup_{B(g_0,R/2)} u\ \leq \ C  \fint_{B(g_0,R)} u \ +\ {C}T(u; g_0,R/2),
\end{equation}
where the ``tail"  is given by
\begin{equation}\label{e:tail def}
T(u;g_0,R)=R^{2s}\int_{| g_0^{-1}\cdot h|>R}\frac {u(h)}{|g_0^{-1}\cdot h|^{Q+2s}}dh.
\end{equation}
\end{thrm}

We note that the potential $V$ in \eqref{e:fracschro} is not
assumed to be radial (i.e., a function of the norm $|\cdot|$), or controlled by a power of $u$. The hypothesis \eqref{e:decay condn} goes
back to the work \cite{Bando89}, see also \cite{V11err} where a similar
assumption was used in the case of Schr\"{o}dinger type equations
modelled on the equations for the  extremals to Hardy-Sobolev
inequalities with polyradial symmetry.
For other results about the Schr\"{o}dinger equation see
\cite{FLS16}.

Our second result is the following  theorem in which we establish the sharp asymptotic decay of weak nonnegative subsolutions to the fractional Yamabe  type equation \eqref{main}. The result applies to weak solutions of $\frlap u=|u|^{2^*(s)-2} u$, since then $|u|$ is a weak subsolution of the Yamabe type equation.

\begin{thrm}\label{t:asympt yamabe}
 Let $\bG$ be a homogeneous group of homogeneous dimension $Q$ and $0<s<1$. If $u\in \mathcal{D}^{s,2}(\bG)$ is a nonnegative subsolution to the nonlocal Yamabe type equation
\begin{equation}\label{e:frac yamabe}
\frlap u=u^{\frac{Q+2s}{Q-2s}},
\end{equation}
then $|\cdot|^{{Q-2s}}\,u\in L^\infty(\bG)$.
\end{thrm}

We mention that in \cite[Theor. 1.1]{BMS16} the authors established, in the setting of $\Rn$, the sharp asymptotic behaviour of the spherically symmetric extremals for the fractional $L^p$ Sobolev inequality, i.e., for the radial nonnegative solutions in $\mathbb{R}^n$  of the equation with critical exponent
\[
(-\Delta_p)^s u =u^{\frac {n(p-1)+sp}{n-sp}},
\]
where $0<s<1$, $1<p<\frac{n}s$, see also \cite{MM19}. However, both \cite{BMS16} and \cite{MM19} use in a critical way the monotonicity and radial symmetry of the solutions  in order to derive the asymptotic behaviour from the regularity of $u$ in the weak space $L^{r,\infty}$, where $r = \frac{n(p-1)}{n-sp}$.
As it is well-known, in the Euclidean setting one can use radially decreasing rearrangement or the moving plane method to establish monotonicity and radial symmetry of solutions to variational problems and partial differential equations. These tools are not available in Carnot groups and proving the relevant symmetries of similar problems remains a very challenging task.

The result of Theorem \ref{t:asympt yamabe} does not rely on the symmetry of the solution, hence the method of proof is new even in the Euclidean setting.
In order to obtain the optimal decay Theorem \ref{t:asympt yamabe} without relying on symmetry of the solution, we use a version of the local boundedness estimate given in Theorem \ref{t:loc bound} and then obtain a new estimate of the tail term, which is particular for the fractional case.

In closing, we provide a brief description of the paper. In Section \ref{S:hom} we introduce the geometric setting of the paper and the relevant definitions. We also prove Proposition \ref{p:Lpregul}, a preparatory result which provides regularity in $L^p$ spaces for subsolutions of fractional Schr\"odinger equations. In Section \ref{s:proof of loc bound} we prove Theorem \ref{t:loc bound}. Finally, in Section \ref{S:main} we prove Theorem  \ref{t:asympt yamabe}.

%%%%%%%%%%%%%%%%%%%%%%%%%%%%%%%%%%%%%%%%%%%%%%%%%%%%%%%%%%%%%

\section{Homogeneous groups and fractional operators}
\label{S:hom}

This section is devoted to providing the necessary background and stating a preliminary result, Proposition \ref{p:Lpregul} below. Let $\bG$ be  a homogeneous group as defined in  \cite[Chapter 1]{FS82}. In particular, $(\bG,\circ)$ is a connected simply connected nilpotent Lie group. Furthermore, the exponential map $\exp:\algg\rightarrow \bG$ is a diffeomorphism of the Lie algebra $\algg$ onto the group $\bG$ and $\algg$ is endowed with a family of non-isotropic group dilations $\delta_\lambda$ for $\lambda>0$. Explicitly, there is a basis $X_j$, $j=1,\dots,n$ of the Lie algebra $\algg$ and positive real numbers $d_j$, such that,
\[
1=d_1\leq d_2\leq\dots\leq d_n \quad \text{and}\quad
\delta_\lambda X_j=\lambda^{d_j} X_j,
\]
which, using the exponential map, define 1-parameter family of automorphisms of the group $\bG$ given by $\exp\circ\ \delta_\lambda\circ\exp^{-1}$. We will use the same notation $\delta_\lambda$ for the group automorphisms. As customary, we indicate with
\[
Q=d_1+\dots+d_n
\]
the homogeneous dimension of $\bG$ with respect to the nonisotropic dilations $\delta_\lambda$.
We will denote with $\volg$ a fixed Haar measure given by the push forward of the Lebesgue measure on the Lie algebra via the exponential map, see \cite[Proposition 1.2]{FS82}. We note that this gives a bi-invariant Haar measure.
 Furthermore, the homogeneous dimension and the Haar measure measure are related by the identity  $d(\delta_\lambda g)=t^Q\, dg$.

The polar coordinates formula for the Haar measure gives the existence of a unique Radon measure $d\sigma(g)$, such that, for $u\in L^1(\bG)$  we have the  identity,  \cite[Prop. (1.15)]{FS82},
\begin{equation}\label{e:polar coordinates integral}
\int_{\bG} u(g)dg=\int_0^\infty\int_{|g|=1} u(\delta_r g) r^{Q-1}d\sigma(g)dr.
\end{equation}
In particular, we have, see \cite{FR66},
\begin{equation}\label{e:polar coord}
\int_{r<|g|<R}|g|^{-\gamma}\volg=\left\{
                  \begin{array}{ll}
                    \frac {\sigma_Q}{Q-\gamma}\left(R^{Q-\gamma} - r^{Q-\gamma} \right), & \gamma\not=Q\\
                    &\\
                    \sigma_Q\log(R/r), & \gamma=Q,
                  \end{array}
                \right.
\end{equation}
where $\sigma_Q = Q \omega_Q$, and $\omega_Q = \int_{B_1} dg>0$.

\subsection{The fractional operator}\label{s:fractional operator}
For $0<s<1$ consider the quadratic form
\[
\mathscr{Q}_s(u,\phi)\overset{def}{=}
\int_{\bG}\int_{\bG} \frac
{(u(g)-u(h))(\phi(g)-\phi(h))}{|g^{-1}\cdot h|^{Q+2s}}
{dg}dh.
\]
Following \cite{GLV22p}, we let $\mathcal{D}^{s,2}(\bG)$ be the  fractional Sobolev space defined as the closure of $C^\infty_0(\bG)$ with respect to the case $p=2$ of the seminorm \eqref{semi}, i.e.
\begin{equation}\label{e:def frac norm}
[u]_{s,2} = \mathscr{Q}_s(u,u)^{1/2} = \left(\int_{\bG}\int_{\bG}\frac {|u(g)-u(h)|^2}{|g^{-1}\cdot h|^{Q+2s}}\, {dg}dh\right)^{1/2}.
\end{equation}
The infinitesimal generator of the quadratic form $\mathscr{Q}_s(u,\phi)$ is the nonlocal operator $\frlap$ defined in \eqref{e:frac L via hypersing} above. By a weak solution  of the equation $\frlap u =F$ we intend a function $u\in \mathcal{D}^{s,2}(\bG)$ such that for any $\phi\in C^\infty_0(\bG)$ one has
\begin{equation}\label{e:weak formulation}
\mathscr{Q}_s(u,\phi) =
\int_{\bG}\int_{\bG} \frac
{(u(g)-u(h))(\phi(g)-\phi(h))}{|g^{-1}\cdot h|^{Q+2s}}\,
{dg}dh=\int_{\bG}F(g)\phi(g){dg}.
\end{equation}
Weak subsolutions are defined by requiring $$\mathscr{Q}_s(u,\phi) \leq \int_{\bG}F(g)\phi(g){dg}$$ for all nonnegative test functions $\phi$.
As shown in  \cite[Theorem 1.1]{GLV22p}, this is equivalent  to defining the fractional operator $\frlap$ by the formula \eqref{e:frac L via hypersing}. As we have underlined in Section \ref{S:intro}, besides the Euclidean case $\bG = \Rn$, in a Lie group of Heisenberg type, equipped with the Koranyi norm, definition \eqref{e:frac L via hypersing} equals, up to a multiplicative constant, the fractional powers of the conformally invariant (or geometric) horizontal Laplacian defined by \eqref{bfmH} above, see \cite{FGMT}, \cite{RT} and \cite{GTaim}.

\subsection{A preparatory result on Lebesgue space regularity}\label{s:proof prop}

In the proof of Theorem \ref{t:loc bound} we will need  the following regularity in Lebesgue spaces involving the fractional operator \eqref{e:frac L via hypersing}. In its proof we adapt the arguments that in the local case were developed in \cite[Lemma
10.2]{GV00}, \cite[Theor. 4.1]{V06} and \cite[Theor. 2.5]{V11}, except that in the nonlocal case one has to use the
Sobolev inequality \eqref{e:fracsobolev}, rather than the
 Folland-Stein embedding $\mathcal{D}^{1,2}(\bG) \hookrightarrow L^{\frac{2Q}{Q-2}}(\bG)$. As far as part (b) of Proposition \ref{p:Lpregul} below is concerned, in addition to the cited references we also mention \cite[Sec. 4]{GL92}, where a similar result was proved for $L^2$ solutions, and \cite[Lemma 2.3]{LU}, for a closely related result concerning the Yamabe equation on the Heisenberg group $\Hn$. In the local case in $\Rn$, a sharp Lorentz space result was obtained for solutions to equations modelled on Yamabe type equations, or more generally for the Euler-Lagrange equation related to the $L^p$ Sobolev inequality. This type of result originated with the work \cite{JS99}, and was subsequently used to obtain the sharp $L^p$ regularity and the asymptotic behaviour  for solutions of such equations,  see \cite[Lemma 2.2]{Vet16} and \cite[Proposition 3.3]{BMS16}. These  results were extended to Yamabe type equations in Carnot groups in { \cite[Theorem 1.1 and Proposition 3.2]{Loi19}}.  We mention that, since we work in the more general setting of a Schr\"odinger type equation, in Proposition \ref{p:Lpregul}  (b) below we do not obtain a borderline $L^{r,\infty}(\bG)$ Lorentz regularity for the considered non-negative subsolutions, instead, we show that  $u\in L^{q}(\bG)$ for  $r=\frac {2^*(s)}{2}<q< \infty$.
For the statement of the next proposition, the reader should keep in mind the definition \eqref{e:r def} of the exponents $r$ and $r'$.

\begin{prop}\label{p:Lpregul} Let $\bG$ be a homogeneous group and suppose that $u\in \mathcal{D}^{s,2}(\bG)$ be a nonnegative subsolution to the nonlocal equation
\begin{equation}\label{e:fracschro1}
\frlap u=Vu,
\end{equation}
with $V\in L^{r'}(\bG)$, i.e., for every $\phi\in C^\infty_0(\bG)$ such that $\phi\ge 0$ one has
\begin{equation}\label{e:weak formulation}
\mathscr{Q}_s(u,\phi) =
\int_{\bG}\int_{\bG} \frac
{(u(g)-u(h))(\phi(g)-\phi(h))}{|g^{-1}\cdot h|^{Q+2s}}\,
{dg}dh\leq \int_{\bG} V(g)u(g)\phi(g){dg}.
\end{equation}
\begin{itemize}
\item[(a)]  We have $u\in L^q(\Om)$ for every $2^*(s)\ \leq\ q\ <\ \infty.$ Furthermore, for any $2^*(s)<q<\infty$ there exist a constant $C_Q>0$, such that for all sufficiently large $M$ for which
\begin{equation}\label{e:basic Lebesgue bound}
\left(\int_{|V|>M} V^{r'}dg \right)^{1/r'}\leq \frac {1}{qC_Q},
\end{equation}
one has
\[
\norm{u}_{L^q(\bG)}\leq (qC_Q M)^{1/q}\norm
{u}_{\mathcal{D}^{s,2}(\bG)}.
\]
\item[(b)] In fact, it holds  $u\in L^{q}(\bG)$ for  $r=\frac {2^*(s)}{2}<q< \infty$.
\item[(c)]  If, in addition,  $V\in   L^{t_0}(\bG)$ for some  $t_0\ >\ r'$, then $u\in L^{q}(\bG)$ for  $r=\frac {2^*(s)}{2}<q\leq \infty$. In addition, the sup of $u$ is estimated as follows,
\[
\norm{u}_{L^\infty(\bG)} \  \leq \ C_Q \, \norm{V}_{L^{t_0}(\bG)}^{\frac {t_0'r}{r-t_0'}} \, \norm{u}_{L^{2^*(s)t_0'}(\bG)},
\]
where $t_0'$ is the H\"older conjugate to $t_0$ and $C_Q$ is a constant depending on the homogeneous dimension.
\end{itemize}
\end{prop}

\begin{proof}
We begin by recalling a few basic facts which are crucial for working with appropriate test functions in the weak formulation of the nonlocal equation \eqref{e:fracschro1}. First, using H\"older's inequality and the definition of the homogeneous fractional Sobolev space $\DsG$, we can take $\phi\in\DsG$  in the weak formulation \eqref{e:weak formulation}.  On the other hand, for a globally Lipschitz function $F$ defined on $\mathbb{R}$ and a  function $u\in\DsG$ we have from \eqref{e:def frac norm} the inequality
$$
[F\circ u]_{s,2}\leq \norm{F'}_{L^\infty(\mathbb{R})} [u]_{s,2}, $$
hence  $F\circ u\in\DsG$. Assuming, in addition, that $F$ is of the form $F(t)=\int_0^t  G'(\tau)^2 d\tau$, then  from Jensen's inequality we have for any nonnegative numbers $a\le b$ the inequality
\[
\left(\frac{G(b) - G(a)}{b-a}\right)^2 = \left(\frac{1}{b-a}\int_a^b G'(\tau)d\tau\right)^2\le \frac{1}{b-a} \int_a^b G'(\tau)^2 d\tau = \frac{F(b)-F(a)}{b-a},
\]
which gives
\begin{equation}\label{FaFb}
(b-a)\left(F(b)-F(a) \right)\geq \left(G(b)-G(a) \right)^2.
\end{equation}
Applying the Sobolev inequality \eqref{e:fracsobolev} to the function $G\circ u$, and using \eqref{FaFb}, we find
\begin{multline}\label{e:basic test fn1}
\norm{G\circ u}_{L^{2^*(s)}(\bG)}^{2}\leq S^2
\int_{\bG}\int_{\bG} \frac {|G\circ u(g)-G\circ u(h)|^2}{|
h^{-1}\cdot g|^{Q+2s}}\, {dg}dh
\\
\leq S^2 \int_{\bG}\int_{\bG} \frac {(u(g)-u(h))(F\circ u(g)-F\circ u(h))}{|g^{-1}\cdot h|^{Q+2s}}\, {dg}dh
\le S^2 \int_{\bG}V(g) u(g) (F\circ u)(g) {dg},
\end{multline}
where in the last inequality we have used \eqref{e:weak formulation} with the choice $\phi = F\circ u$ as a test function.

For the proof of parts {(a) and (c)}  see for example \cite[Lemma
10.2]{GV00} and \cite[Theorem 4.1]{V06}, but one has to use the fractional Sobolev inequality \eqref{e:fracsobolev} rather than
the Folland-Stein inequality. We give the proof of part {(b)} below taking into account also \cite[Proposition 3.3]{BMS16} which dealt with the Euler-Lagrange equation of the fractional $p$-Laplacian in the Euclidean setting.

To prove (b) in Proposition \ref{p:Lpregul} we will show that for any $0<\alpha<1$ we have that $u\in L^{r(1+\alpha)}(\bG)$. From part (a) and the fact that $2^*(s)/2 < r(1+\alpha)<2r=2^*(s)$ the claim of part (b) will be proven. The details are as follows.
For  $\ve>0$ and $0<\alpha<1$,  consider the functions
\[
F_\ve(t)=\int_0^t G'_\ve(\tau)^{2}d\tau,\ \ \ \ \text{where}\ \ G_\ve(t)=t(t+\ve)^{(\alpha-1)/2}.
\]
Notice that $F_\ve$ is nondecreasing by definition. A simple calculation shows that
\begin{equation}\label{simple}
0 \le G'_\ve(t) = \frac {1}{(t + \eps)^{(3-\alpha)/2}}\left [\frac {1+\alpha}{2} t + \eps \right] \leq \frac {1}{(t + \eps)^{(1-\alpha)/2}}\leq \frac {1}{ \eps^{(1-\alpha)/2}},
\end{equation}
where we have used that $\alpha<1$.
This shows in particular that $F'_\eps(t) = G'_\ve(t)^2 \le \ve^{\alpha-1}$, therefore $F_\eps$ is a globally Lipschitz function. We thus find from \eqref{e:basic test fn1}
\begin{equation}\label{Ge}
\norm{G_\ve\circ u}_{L^{2^*(s)}(\bG)}^{2}\leq S^2\int_{\bG}|V| uF_\eps(u)dg.
\end{equation}
In order to estimate the right-hand side in \eqref{Ge} we will use the following inequalities, which are valid for $u\ge 0$,
\begin{equation}\label{e:prop of f and G}
F_\eps(u)\leq \frac {1}{\alpha}u^\alpha\qquad \text{and} \qquad uF_\eps(u)\leq \frac {1}{\alpha} G_\ve(u)^2.
\end{equation}
The former is easily proved by noting that
\[
F_\eps(u)\leq \int_0^u \frac {dt}{(t + \eps)^{1-\alpha}} \le \frac{(u+\ve)^{-\alpha}}{\alpha} \leq \frac {1}{\alpha}u^\alpha.
\]
This estimate trivially gives
$u F_\ve(u)  \le u \frac{(u+\ve)^\alpha - \ve^\alpha}{\alpha}$, and therefore from the definition of $G_\ve$ we see that the latter inequality in \eqref{e:prop of f and G} does hold provided that
\[
(u+\ve)^\alpha - \ve^\alpha \le \frac{u}{u+\ve}(u+\ve)^{\alpha},\quad \text{i.e.,}\quad  (u+\ve)^{\alpha-1} - \ve^{\alpha-1} \leq 0.
\]
The latter inequality follows from the trivial inequality
\[
(1+x)^{1-\alpha}\geq 1
\]
valid for $x\geq 0$ and $0<\alpha<1$.

Keeping in mind the definition \eqref{e:r def} of the exponents $r$ and $r'$, using now in \eqref{Ge} the first inequality in \eqref{e:prop of f and G} and H\"older inequality, we easily obtain for a fixed $\delta>0$
\begin{multline}\label{e:RHS1}
\norm{G_\ve\circ u}_{L^{2^*(s)}(\bG)}^{2}\leq S^2\left( \frac {1}{\alpha}\int_{\{|V|>\delta\}}|V|u^{1+\alpha}dg + \int_{\{|V|\leq \delta\}}|V|u F_\eps(u)dg\right)\\
\leq S^2 \left[ \frac {1}{\alpha} \left(\int_{\{|V|>\delta\}}|V|^{r'}dg\right)^{1/r'} \left(\int_{\{|V|>\delta\}}u^{r(1+\alpha)}dg \right)^{1/r} \right] \\
+ S^2\left[ \left(\int_{\{|V|\leq \delta\}}|V|^{r'}dg\right)^{1/r'} \left(\int_{\{|V|\leq \delta\}}(uF_\eps(u))^{r}dg \right)^{1/r}\right].
\end{multline}
Next, we use the second of the inequalities \eqref{e:prop of f and G} to obtain the estimate
\begin{equation}\label{e:RHS2}
 \left(\int_{\{|V|\leq \delta\}}(uF_\eps(u))^{r}dg \right)^{1/r}\leq \frac {1}{\alpha}\left(\int_{\{|V|\leq \delta\}}\left(G_\eps(u)\right)^{2r} \right)^{1/r}\leq \frac {1}{\alpha}\norm{G_\ve\circ u}_{L^{2^*(s)}(\bG)}^{2}.
\end{equation}
By Lebesgue dominated convergence one has $\int_{\{|V|\leq \delta\}}|V|^{r'}dg\longrightarrow 0$ as $\delta\to 0^+$. Therefore, we can choose $\delta>0$ so small that
\begin{equation}\label{delta}
\frac {S^2}{\alpha}\left(\int_{\{|V|\leq \delta\}}|V|^{r'}dg\right)^{1/r'} <\frac 12.
\end{equation}
Combining \eqref{delta} with \eqref{e:RHS2} we can absorb in the left-hand side the second term in the right-hand side of \eqref{e:RHS1}, obtaining the inequality
\begin{equation}\label{e:RHS3}
\norm{G_\ve\circ u}_{L^{2^*(s)}(\bG)}^{2}\leq \frac {2S^2}{\alpha}\left(\int_{\{|V|>\delta\}}|V|^{r'}dg\right)^{1/r'} \left(\int_{\{|V|>\delta\}}u^{r(1+\alpha)}dg \right)^{1/r}.
\end{equation}
Notice that the hypothesis $V\in L^{r'}(\bG)$ and Chebyshev inequality imply that the distribution function of $V$ satisfies for every $\delta>0$
\begin{equation}\label{cheb}
\mu(\delta)=\left\vert\{g\in \bG\mid |V(g)|>\delta \}\right\vert \leq \frac {1}{\delta^{r'}} \int_{\{|V|>\delta\}} |V|^{r'} dg < \infty.
\end{equation}
Since $r(1+\alpha)<2r=2^*(s)$, H\"older inequality thus gives
\[
\left(\fint_{\{|V|>\delta\}} u^{r(1+\alpha)}dg\right)^{\frac {1}{r(1+\alpha)}}\leq \left(\fint_{\{|V|>\delta\}} u^{2^*(s)}dg\right)^{\frac {1}{2^*(s)}},
\]
or equivalently, recalling that $2r=2^*(s)$,
\begin{multline}\label{aftercheb}
\left(\int_{\{|V|>\delta\}} u^{r(1+\alpha)}dg\right)^{\frac {1}{r}}  \leq \mu(\delta)^{\frac{1}{r}-\frac {1+\alpha}{2r}} \left(\int_{\{|V|>\delta\}} u^{2^*(s)}dg\right)^{\frac {1+\alpha}{2^*(s)}}\\
=\mu(\delta)^{\frac {1-\alpha}{2r}} \left(\int_{\{|V|>\delta\}} u^{2^*(s)}dg\right)^{\frac {1+\alpha}{2^*(s)}}.
\end{multline}
Using \eqref{cheb}, \eqref{aftercheb} and $r'/r=1/(1-r)$ in \eqref{e:RHS3} we come to
\[
\norm{G_\ve\circ u}_{L^{2^*(s)}(\bG)}^{2}\leq \frac {2S^2}{\alpha} \frac {\norm{V}_{L^{r'(\bG)}}^{1+\frac{1-\alpha}{2(1-r)}}}{\delta^{\frac{1-\alpha}{2(1-r)}}} \norm{u}^{1+\alpha}_{L^{2^*(s)}(\bG)}.
\]
Letting $\eps$ go to 0, noting that $\lim_{\eps\rightarrow 0}G_\eps(u)=u^{(1+\alpha)/2}$ brings us to
\[
\left( \int_{\bG}u^{(1+\alpha)r}\right)^{1/r}<\infty,
\]
which, taking into account also part (a), completes the proof of part (b) of Proposition \ref{p:Lpregul}.

\end{proof}

\section{Proof of Theorem \ref{t:loc bound}}\label{s:proof of loc bound}

The proof consists of several steps detailed in the following sub-sections.

\subsection{The localized fractional Sobolev inequality}
The proof of Theorem \ref{t:loc bound} will use the following
version of a localized fractional Sobolev inequality. For an open
set $\Omega \subset \mathbb{G}$ we denote by
$\mathcal{D}^{s,2}(\Omega)$ the completion of $C_0^\infty(\Omega)$
with respect to the norm
\begin{equation*}
\norm{v}_{\mathcal{D}^{s,2}(\Omega)}=\left(\int_{\mathbf{G}}\int_{\mathbf{G}}
\frac{|\tilde v(g)-\tilde v(h)|^2}{|g^{-1}\cdot h|^{Q+2s}}\right)^{1/2},
\end{equation*}
where $\tilde v$ denotes the extension of $v$ to a function on $\bG$, which is equal to zero outside of $\Omega$.
\begin{lemma}\label{l:localized frac sobolev}
Let $0<s<1$ and $2s<Q$. There exists a constant $C=C(Q,s)>0$ such
that, for any ball $ B_R$ of radius  $R$, $r<R$, and  $v\in
\mathcal{D}^{s,2}(B_R)$  with $\mathrm{supp}\, v\subset B_{r }$ we
have
\begin{equation}\label{e:loc Sobolev}
\left[\int_{B_R}|v|^{2^*(s)} \vol\right]^{2/2^*(s)}\leq C\left [
\int_{B_R}\int_{B_R} \frac{| v(g)-v(h)|^2}{|g^{-1}\cdot h|^{Q+2s}}dgdh\ +\ \frac {1}{R^{2s}}\left(\frac
{R}{R-r}\right)^{Q+2s}\int_{B_R}|v|^2 \vol\right],
\end{equation}
\end{lemma}
\begin{proof}
The proof is essentially contained in the Euclidean version  \cite[Proposition 2.3]{BP15}.
We will use the trivial extension  and then apply the fractional Sobolev inequality \eqref{e:fracsobolev}. Since $v$ has compact support in $B_R$ its extension  by zero  on the complement of the ball is a function $\tilde v\in \mathcal{D}^{s,2}(\bG)$. Furthermore, due to the assumption on the support of $v$, we have
\begin{multline}\label{e:local sobolev ineq proof}
\norm{\tilde v}^2_{\mathcal{D}^{s,2}(\bG)}\leq \int_{B_R}\int_{B_R} \frac{|v(g)-v(h)|^2}{|g^{-1}\cdot h|^{Q+2s}}dgdh + 2 \int_{B_{r }}|v(g)|^2\int_{\bG\setminus B_R} \frac {1}{|g^{-1}\cdot h|^{Q+2s}} dh\, dg\\
\leq \int_{B_R}\int_{B_R} \frac{|v(g)-v(h)|^2}{|g^{-1}\cdot h|^{Q+2s}}dgdh + 2 \int_{\bG\setminus B_R} \sup_{g\in B_{r}}\frac {1}{|g^{-1}\cdot h|^{Q+2s}} dh\, \int_{B_R}|v(g)|^2 dg\\
\leq \int_{B_R}\int_{B_R} \frac{|v(g)-v(h)|^2}{|g^{-1}\cdot
h|^{Q+2s}}dgdh +  C\left(\frac {R}{R-r}\right)^{Q+2s}\frac
{1}{R^{2s}}\, \int_{B_R}|v(g)|^2 dg.
\end{multline}
In order to see the last of the above inequalities we used polar coordinates as in the identity \eqref{e:polar coord} to obtain the following inequalities, where $\sigma_Q$ is the area of the unit sphere,
\begin{multline*}
\int_{\bG\setminus B_R} \sup_{g\in B_{r}}\frac {1}{|g^{-1}\cdot h|^{Q+2s}} dh \leq \int_{B_{2R}\setminus B_R} \frac {1}{\left(|h|-r\right)^{Q+2s}} dh + \int_{\bG\setminus B_{2R}} \frac {1}{|g^{-1}\cdot h|^{Q+2s}} dh\\
\leq \sigma_Q\left[ \int_R^{2R} \frac {t^{Q-1}}{\left(t-r\right)^{Q+2s}} dt + \int_{2R}^\infty \left(\frac {2}{t}\right)^{Q+2s} t^{Q-1}dt\right]
\end{multline*}
since $|g^{-1}\cdot h|\geq |h|-r\geq |h|/2$ when $r<R$, $|g|<r$ and $|h|\geq 2R$. Therefore we have
\[
\int_{\bG\setminus B_R} \sup_{g\in B_{r}}\frac {1}{|g^{-1}\cdot h|^{Q+2s}} dh\leq C\left[ \frac {R^Q}{Q+2s} \left(\frac {1}{R-r} \right)^{Q+2s}+ \frac {1}{2s}\frac {2^{Q+2s}}{R^{2s}}\right],
\]
which gives  \eqref{e:local sobolev ineq proof}. The latter implies trivially  \eqref{e:loc Sobolev}.
\end{proof}

\subsection{Caccioppoli inequality}
We  begin by stating the adaptation to our setting of the Caccioppoli inequality  for the fractional $p-$Laplacian in Euclidean space \cite{BP15}.  For $\beta\geq 1$ and $\delta>0$ define the following functions for $t\geq 0$,
\begin{equation}\label{e:g fns}
\phi(t)=(t+\delta)^\beta, \qquad \Phi(t)=\int_0^t|\phi'(\tau)|^{1/2} \, d\tau =2\frac{\beta^{1/2}}{\beta+1}(t+\delta)^{(\beta+1)/2}.
\end{equation}
For our goals, the precise value of $\delta$  is given in  \eqref{e:set delta} below.
Suppose $\Omega'\Subset \bG$ and  $\psi\in\COG$ is a positive function with supp$\, \psi\subset\Omega'$. Let $u$ be a weak nonnegative subsolution to the equation $\frlap u =F$ with $F\in L^{(2^*(s))'}$. Then, we have for some constant $C=C(Q)$, which is independent of $\Omega'$, the inequality
\begin{equation}\label{e:rough cacc}
\begin{split}
\int_{\Omega'}\int_{\Omega'} &\frac{\Big|\Phi(u(g))\,\psi(g)-\Phi(u(h))\,\psi(h)\Big|^{2}\,}{|g^{-1}\cdot h|^{Q+2s}}\leq C\int_\bG |F|\,\phi(u(g))\, \psi^2(g)\, dg\\
&+ \frac{C}{\beta}\,\left(\frac{\beta+1}{2}\right)^2\,\int_{\Omega'}\int_{\Omega'} \frac{|\psi(g)-\psi(h)|^2}{|g^{-1}\cdot h|^{Q+2s}}\,\Big((\Phi(u(g))^2+\Phi(u(h))^2\Big)\, dg\,dh\\
&+C\,\left(\sup_{h\in \mathrm{supp}\,\psi} \int_{\bG\setminus \Omega'} \frac{|u(g)|}{|g^{-1}\cdot h|^{Q+2s}}\,dg\right)\, \int_{\Omega'} \phi(u)\, \psi^2\, dg.
\end{split}
\end{equation}
The proof of \eqref{e:rough cacc} follows by adapting to our setting, for the case $p=2$, that of the  localised Caccioppoli inequality in $\Rn$ in \cite[Proposition 3.5]{BP15}.

We now fix $g_0\in\bG$, and for $0<r<R$, we take a  nonnegative smooth bump function $\psi\in\COG$ such that
\begin{equation}\label{e:psi choice}
\psi\vert_{B_r(g_0)}\equiv 1,\qquad \mathrm{supp}\, \psi\Subset
B_{\frac {r+R}{2}}(g_0), \qquad
|\psi(g_1)-\psi(g_2)|\leq \frac{C}{R-r} |g_1^{-1}\cdot g_2| .% |\nabla \psi|\leq \frac{C}{R-r}.
\end{equation}
In order to achieve  \eqref{e:psi choice} we take a cut-off function $\psi(g)=\eta (|g_0^{-1}\cdot g|)$, where $\eta$ is a smooth bump function on the real line, such that, $\eta(t)\equiv 1 $ on $|t|\leq r $, $\eta\equiv 0$ on $t\geq (R+r)/2$ and for some constant $K>0$ we have $|\eta'(t)|\leq K/(R-r)$ for all $t$. Hence, for any $\rho_1,\, \rho_2\in \mathbb{R}$ we have
$$|\eta(\rho_1)-\eta(\rho_2)|\leq  \frac {K}{R-r}|\rho_1-\rho_2|.$$
On the other hand, if we let $\rho(g)=| g^{-1}\cdot g_o|$, then from the triangle inequality \eqref{e:triangle ineqs} it follows that $\rho$ is a Lipschitz continuous function  with respect to the gauge distance, with Lipschitz constant equal to 1,
\[
|\rho(g_1)-\rho(g_2)|\leq |g_1^{-1}\cdot g_2|, \qquad g_1, g_2\in \bG
\]
Therefore,  for  $g_j\in \bG$ and $\rho_j=\rho(g_j)$, $j=1,2$, we have
\[
|\psi(g_1)-\psi(g_2)|=|\eta(\rho_1)-\eta(\rho_2)|\leq \frac {K}{R-r} |\rho_1-\rho_2|\leq  \frac {K}{R-r}|g_1^{-1}\cdot g_2|.
\]
For the remainder of the proof, for any $r>0$  we will denote by $B_r$ the ball $B_r(g_0)$ with the understanding that the center is the fixed point $g_0$.

If $u$ is a nonnegative  weak subsolution to
\[
\frlap u=Vu,
\]
then, with the above choice of $\psi$ and $F=Vu$, \eqref{e:rough cacc}
implies the following inequality
\begin{multline}\label{e:cacciopolli1}
\int_{B_R}\int_{B_R}\frac
{|u_\delta(g)^{(\beta+1)/2}\psi(g)-u_\delta(h)^{(\beta+1)/2}\psi(h)|^2}{|g^{-1}\cdot
h|^{Q+2s}} \, dg\, dh \\
\leq C \beta \left[
\int_{B_R}\psi^2 V u_\delta^{\beta+1}\vol\ + \ \left(\frac {R}{R-r} \right)^{2} \frac {1}{R^{2s}} \int_{B_R}u_\delta^{\beta+1}\vol %\right.\\
+  \left(\frac {R}{R-r}\right)^{Q+2s}  \frac {1}{R^{2s}}
T(u;g_0,R) \int_{B_R}u_\delta^{\beta} \vol\right],
\end{multline}
where $u_\delta=u+\delta$ and $T(u;g_0,R)$ is the tail \eqref{e:tail def}. The proof of \eqref{e:cacciopolli1} is contained in \cite[Theorem 3.8 and (3.29)]{BP15}, except we have to use the Lipschitz bound in \eqref{e:psi choice} for the term $|\psi(g)-\psi(h)|^2$ in \eqref{e:rough cacc}.

Next,  we apply to the inequality \eqref{e:cacciopolli1}  the localized Sobolev
inequality \eqref{e:loc Sobolev}, with $r$ replaced with $(R+r)/2$
and $v$ with $\psi \,u_\delta^{(\beta+1)/2}$, taking into account

\[
 \left(\frac {R}{R-r} \right)^2 \leq \left(\frac {R}{R-r}\right)^{Q+2s}
\]
and also that by the choice of $\psi$ we have  $\mathrm{supp}\,
(\psi\,u_\delta^{(\beta+1)/2})\Subset B_{(R+r)/2}\Subset B_R$  . As
a result, we obtain
\begin{multline}\label{e:cacciopolli2}
\left[ \int_{B_{\frac {R+r}{2}}}\psi^{2^*(s)}u_\delta^{(\beta+1)2^*(s)/2 } \vol\right]^{2/2^*(s)}\leq
C\beta \left[\int_{B_R}\psi^2Vu_\delta^{\beta+1}\vol\right.\\
\left. + \left(\frac {R}{R-r}\right)^{Q+2s}\frac {{1}}{R^{2s}}
\int_{B_R}u_\delta^{\beta+1}\vol
 + \left(\frac {R}{R-r}\right)^{Q+2s}\frac {1}{R^{2s}}T(u;g_0,R) \int_{B_R}u_\delta^{\beta}\vol\right].
\end{multline}

\subsection{Use the assumptions on $V$}
This is the core of the new argument leading to our result.
By H\"older's inequality and $\mathrm{supp} \,\psi\Subset B_{(R+r)/2}$, we have
\[
\int_{B_R}\psi^2Vu_\delta^{\beta+1}\vol\ \leq \ \left( \int_{B_R}V^{t_0}\vol\right)^{1/t_0} \left(  \int_{B_{\frac {R+r}{2}}} \left(\psi\, u_\delta^{\frac {\beta+1}{2}}\right)^{2^*(s)}\vol\right)^{1/t} \left( \int_{B_{\frac {R+r}{2}}} \left(\psi\, u_\delta^{\frac {\beta+1}{2}} \right)^2 \vol \right)^{1/\kappa},
\]
where
\[
t=\frac {2st_0}{Q-2s}\qquad \text{and} \qquad \kappa=\frac {2st_0}{2st_0-Q}
\]
so that
\[
\frac {1}{t_0}+\frac 1t+\frac 1{\kappa}=1\qquad \text{and}\qquad \frac {2^*(s)/2}{t}+\frac {1}{\kappa}=1,
\]
which is possible due to the assumptions in Theorem \ref{t:loc bound}. Next, we use Young's inequality $ab\leq \varepsilon \frac {a^{\kappa'}}{\kappa'}+\frac {1}{\varepsilon^{\kappa-1}}\frac {b^\kappa}\kappa$ in the right-hand side of the above inequality to conclude
\[
\int\psi^2Vu_\delta^{\beta+1} \vol \leq \frac {\varepsilon}{\kappa'}\left(  \int_{B_{\frac {R+r}{2}}} \left(\psi\, u_\delta^{\frac {\beta+1}{2}}\right)^{2^*(s)}\vol \right)^{\kappa'/t}\ + \ \frac {1}{\varepsilon^{\kappa-1}\kappa}\left( \int_{B_R}V^{t_0}\vol\right)^{\kappa/t_0}\left( \int_{B_{\frac {R+r}{2}}} \psi^2\, u_\delta^{\beta+1}\vol \right).
\]
Hence, taking into account $\kappa'/t=2/2^*(s)$, $\kappa/t_0=\frac {2s}{2st_0-Q} $ and the above inequality together with the properties of $\psi$, we obtain from  \eqref{e:cacciopolli2} the following inequality
\begin{multline*}
\left[ \int_{B_{\frac {R+r}{2}}} \psi^{2^*(s)}u_\delta^{(\beta+1)\frac {2^*(s)}{2} } \vol\right]^{\frac {2}{2^*(s)}} \leq  \beta \left \{\frac {C\varepsilon}{\kappa'}\left[ \int_{B_{\frac {R+r}{2}}}\psi^{2^*(s)} u_\delta^{(\beta+1)\frac {2^*(s)}{2} }\vol \right]^{\frac {2}{2^*(s)}} \right.\\
 +\frac {C}{\varepsilon^{\kappa-1}\kappa}\left( \int_{B_R}V^{t_0}\vol\right)^{\frac {2s}{2st_0-Q}}\left( \int_{B_{R}} u_\delta^{\beta+1} \vol\right)\\
\left.+C\left(\frac {R}{R-r}\right)^{Q+2s}\left[ \frac {1}{R^{2s}}
\int_{B_R}u_\delta^{\beta+1}\vol
 + \frac {1}{R^{2s}}T(u;g_0,R)
 \int_{B_R}u_\delta^{\beta}\vol\right] \right \} .
\end{multline*}
Choosing $\varepsilon$ such that $\frac {C\varepsilon\beta}{\kappa'}=\frac
12$, we absorb the first term in the right-hand side in the
left-hand side, and then reduce the domain of integration, taking
into account that $\psi\equiv 1$ on $B_r$, which brings us  to the
following inequality
\begin{multline*}
\left[ \int_{B_{r}}u_\delta^{(\beta+1)\frac {2^*(s)}{2} }\vol
\right]^{\frac {2}{2^*(s)}}\leq C\left[\beta^{\kappa} \left
(\int_{B_R}V^{t_0}\vol\right)^{\frac{2s}{2st_0-Q}} +\left(\frac
{R}{R-r}\right)^{Q+2s}\frac {{\beta}}{R^{2s}} \right]
\int_{B_R}u_\delta^{\beta+1}\vol
\\
 + C\left(\frac {R}{R-r}\right)^{Q+2s}\frac {\beta}{R^{2s}}T(u;g_0,R) \int_{B_R}u_\delta^{\beta}\vol.
\end{multline*}
Since  $u_\delta^\beta\leq u_\delta^{\beta+1}/\delta$, the above
inequality allows us to conclude
\begin{multline*}
\left[ \int_{B_{r}}u_\delta^{(\beta+1)\frac {2^*(s)}{2} }\vol
\right]^{\frac {2}{2^*(s)}}\leq C\beta^{\kappa}
\left[\left(\int_{B_R}V^{t_0}\vol\right)^{\frac{2s}{2st_0-Q}}
+\left(\frac {R}{R-r}\right)^{Q+2s}\frac {1}{ R^{2s}} \right.
\\
\left. + \left(\frac {R}{R-r}\right)^{Q+2s}\frac {1}{\delta
R^{2s}}T(u;g_0,R) \right] \int_{B_R}u_\delta^{\beta+1}\vol.
\end{multline*}
We recall that in the latter inequality we have radii $0<r<R$ and all balls are centered at the fixed point $g_0$.  Suppose, in addition, that $ 2R_0= {|g_0|}$  and $0<R\leq R_0$. Then, we have   $$B_R=B(g_0,R)\subset B(g_0,R_0)\subset \bG\setminus B(0,R_0),%=\{ |g|>R_0\}
$$
taking into account the triangle inequality \eqref{e:triangle ineqs}.
Therefore, for $R_0\geq 2\bar R_0$  the decay assumption of $V$, cf. Theorem \ref{t:loc bound}\ref{a:decay of V}), and the above inclusions imply that  for some  constant $C= C(Q,s, K_0)$ we have the bound
\[
\left(\int_{B_R}V^{t_0}\vol\right)^{2s/(2st_0-Q)}\leq \left(\int_{B_{R_0}}V^{t_0}\vol\right)^{2s/(2st_0-Q)}\leq \left(\int_{\{
|g|>R_0\} }V^{t_0}\vol\right)^{2s/(2st_0-Q)} \leq \frac {C}{R_0^{2s}}\leq \frac {C}{R^{2s}}
\]
after using $0<R\leq R_0$ for the last inequality.
Therefore, also observing that $R/(R-r)>1$, we have proven that there exists a constant $C=
C(Q,s, K_0)$, such that for any $\beta\geq 1$,  $g_0$ such that  $ R_0=\frac {|g_0|}{2}\geq \bar R_0$,  and radii $0<r<R\leq R_0$ we have
\begin{equation}\label{e:ineq to start Moser}
\left[ \int_{B_{r}}u_\delta^{(\beta+1)\frac {2^*(s)}{2} }\vol
\right]^{\frac {2}{2^*(s)}}\leq \frac {C\beta^{\kappa}}{R^{2s}}
\left(\frac {R}{R-r}\right)^{Q+2s}\left[1 +\frac
{T(u;g_0,R)}{\delta } \right] \int_{B_R}u_\delta^{\beta+1}\vol.
\end{equation}

\subsection{Moser's iteration}  %Before we turn to the proof we note that  Proposition \ref{p:Lpregul}  implies that $u\in L^{q}(\bG)\cap L^\infty(\bG)$, for any $q>2^*(s)/2$.

By  Proposition \ref{p:Lpregul} c) we have that $u\in L^{q}(\bG)\cap L^\infty(\bG)$ for any $q\geq 2^*(s)/2$, hence $u\in L^{q_0}_{loc}(\bG)$ for  $q_0\geq 2$. In fact, for the proof of the theorem we can assume $q_0=2$, but the argument is valid for any $q_0\geq 2$. We also let $\beta=q_0 -1$.

Recalling that the exponent  $r=2^*(s)/2=Q/(Q-2s)>1$, see \eqref{e:r def}, we define the sequence
\[
 q_{j+1}=r q_j>q_j, \quad j=0,1,2,\dots.
\]
From \eqref{e:ineq to start Moser} we have with $B_{r_{j}}=B(g,r_j)$, $r_j=\frac {R}{2}(1+2^{-j})$, $j=0,1,2,\dots$ the inequality
\begin{multline*}
\left(
\int_{B_{r_{j+1}}}u_\delta^{q_{j+1}}\vol\right)^{1/q_{j+1}}\ \leq
\ \frac {C^{ 1/{q_{j}}}q_j^{\kappa/{q_j}}} {r_j^{{2s}/{q_{j}}}}
\left[\left( \frac
{r_j}{r_j-r_{j+1}}\right)^{Q+2s}\right]^{1/q_{j}}\left[1+ \frac
{T(u,g_0,r_j)}{\delta}\right]^{1/q_{j}} \left( \int_{B_{r_{j}}}
u_\delta^{q_{j}}\vol\right)^{1/{q_{j}}}.
\end{multline*}
The definition of the tail \eqref{e:tail def} gives for a fixed  $R\ge \bar R_0$ and $R/2\le r_j<R$ the inequality
$$T(u;g_0,r_j)={r_j}^{2s}\int_{| g^{-1}\cdot h|>r_j}\frac {u(h)}{| g^{-1}\cdot h|^{Q+2s}}dh \leq 2^{2s}T(u;g_0,R/2)$$
while a simple estimate shows
\[
 \left(\frac {r_j}{r_j-r_{j+1}}\right)^{Q+2s} \leq 2^{(j+2)(Q+2s)}.
\]
Hence, letting
\[
M_j=\left( \int_{B_{r_{j+1}}}u_\delta^{q_{j+1}}\vol\right)^{1/q_{j+1}}\qquad\text{and}\qquad T=1+\frac {T(u;g,R/2)}{\delta}
\]
we have with some constants $C_0$ and $C_1$ depending on $Q$ and $s$ the inequality

\[
M_{j+1}\leq \frac {C_0^{
1/q_{j}}C_1^{(j+2)/q_{j}}(q_j)^{\kappa/q_j}}{R^{2s/q_{j}}}\, T^{
1/q_{j}}\, M_j.
\]
Therefore, for
\begin{equation}\label{e:set delta}
\delta=T(u;g,R/2),
\end{equation}
we have $T=2$ and we come to the inequality
\[
M_{j+1}\ \leq \ \frac {C^{ (j+2)/{q_{j}}}(q_j)^{\kappa/q_j}}
{R^{{2s}/{q_{j}}}} M_j.
\]
Therefore, recalling that $q_{j+1}=r^{j} q_0$, we obtain

\[
\sup_{B(g,R/2)} (u+\delta)\ \leq \ C \left(\frac {1}{R^{2s}}
\right)^{\frac {1}{q_0}\sum_{j=0}^\infty r^{-j} }\, C^{\frac
{1}{q_0}\sum_{j=0}^\infty (j+2)r^{-j}}\prod_{j=0}^\infty
(q_j)^{\kappa/q_j} \, M_0.
\]
From the definitions of the exponents $r$ and its H\"older
conjugate $r'=Q/(2s)$ we have  $$\frac {1}{q_0}\sum_{j=0}^\infty
\frac {1}{r^{j}}=\frac {1}{q_0}\frac {r}{r-1}=\frac {r'}{q_0}
=\frac {Q}{2sq_0}$$ and
$$
\prod_{j=0}^\infty (q_j)^{\kappa/q_j}<\infty \quad
\mbox{since}\quad \sum_{i=1}^\infty \frac{\log(q_j)}{q_j}<\infty.
$$
Thus, we come to
\[
\sup_{B(g_0,R/2)} (u+\delta)\ \leq \ C \left (\fint_{B(g_0,R)} (u+\delta)^{q_0}\vol \right)^{1/{q_0}}
.
\]
If we let $q_0=2$  and take into account  the definition of $\delta$ we have shown that there is  a constant  $C_0$, depending on $Q$ and $s$,  such that the inequality
\begin{equation}\label{e:loc bound2}
\sup_{B(g_0,R/2)} u\ \leq \ C_0 \left (\fint_{B(g_0,R)} u^{2} \vol \right)^{1/{2}}\ +\ {C_0}T(u; g_0,R/2).
\end{equation}
holds  for any $g_0\in \bG$ with $|g_0|=2R_0\geq 2\bar R_0$ and $0<R\leq R_0$, where $\bar R_0$ is the radius in the assumptions of Theorem \ref{t:loc bound}.

\subsection{Lowering the exponent}
 To lower the exponent in the average integral in the above inequality we follow the standard argument, see for example \cite[p. 223 Theorem 7.3]{Giusti03}, except we need to account for the tail term similarly to \cite[Corollary 2.1]{KMS15}. In view of the eventual use of the sought estimate in obtaining the asymptotic behaviour of the solution, it is also important to keep the constant in the inequality independent of $R$ as in \eqref{e:loc bound2}.
For any $\rho>0$ let  $$M_\rho\overset{def}{=}\sup_{B(g_0,\rho)}u.$$ First, we will show the
following slight modification of \eqref{e:loc bound2}. There is  a constant  $C_1$, depending  on $Q$ and $s$, such that, for all $g_0\in \bG$ with $|g_0|=2R_0\geq 4\bar R_0$ and $0<r<R\leq R_0$ we have
\begin{equation}\label{e:sup r and R}
M_r\leq C_1 \left(\frac {R}{R-r}\right)^{Q}\left[\left(\fint_{B(g_0,R)}
u^{2} \vol \right)^{1/2} + T(u; g_0, R)\right].
\end{equation}
Letting $\tau=r/R$, $0<\tau<1$, the above inequality is equivalent to showing, with the same constant $C_1$, that we have
\begin{equation}\label{e:lower exp tau}
M_{\tau R}\leq  \frac {C_1}{(1-\tau)^{Q}}\left[\left(\fint_{B(g_0,R)}
u^{2}  \vol\right)^{1/2} \ +\
T(u; g_0, R)\right].
\end{equation}
We turn to the proof of \eqref{e:lower exp tau}.  Let $g_1\in B(g_0,\tau
R)$ and $\rho $ be sufficiently small, in fact,
$$\rho=\frac {(1-\tau)R}{4},$$ so that,
\[
B(g_1,\rho)\subset B(g_1,2\rho)\subset B(g_0, R)\quad\text{ and }\quad \sup_{B(g_0,\tau R)} u = \sup_{B(g_1,\rho)} u.
\]
Notice that by the triangle inequality  we have $|g_1|\geq 2\bar R_0$, which follows from $|g_0|=2R_0\geq 4\bar R_0$, cf. the line above \eqref{e:sup r and R}, hence we can apply \eqref{e:loc bound2} to the ball $B(g_1,\rho)$, which gives
\begin{multline}\label{e:sup est tau 1}
M_{\tau R} = \sup_{B(g_1,\rho)} u\leq \ C_0 \left (\fint_{B(g_1,2\rho)} u^{2} \vol \right)^{1/{2}}\ +\ {C_0}T(u; g_1,\rho)\\
\leq C_02^{-Q/2}\left(\frac {R}{\rho} \right)^{Q/2}\left( \fint_{B(g_0,R)} u^{2}\vol\right)^{1/2}\ +\ {C_0}T(u; g_1,\rho)\\
=\frac {C_02^{Q/2}}{(1-\tau)^{Q/2}} \left( \fint_{B(g_0,R)} u^{2}\vol\right)^{1/2}\ +\ {C_0}T(u; g_1,\rho),
\end{multline}
taking into account that by the definition of $\rho$ we have  $R/\rho=\frac {4}{1-\tau}$.
 We will estimate the tail term in the last line by using the tail term centered at $g_0$ and radius  $R$,  and the average of $u$ over the ball $B(g_0,R)$. For this we split the domain of integration of the integral in the formula for the tail,
\[
T(u; g_1,\rho)=\rho^{2s}\int_{\bG\setminus B(g_1,\rho)} \frac{u(h)}{|h^{-1}\cdot g_1|^{Q+2s}}\, dh,
\]
in two disjoint sets
\[
{\bG\setminus B(g_1,r)}=\left( \bG\setminus B(g_0,R)\right)\cup \left( B(g_0,R)\setminus B(g_1,\rho)\right).
\]
The integral over the second of the above sets is estimated by using $h\notin B(g_1,\rho)$, followed by  H\"older's inequality, to obtain
\begin{multline}\label{e:tail 1st}
\rho^{2s}\int_{B(g_0,R)\setminus B(g_1,\rho)} \frac{u(h)}{|h^{-1}\circ g_1|^{Q+2s}}\, dh \leq \frac{\rho^{2s}}{\rho^{Q+2s}}\int_{B(g_0,R)\setminus B(g_1,\rho)}u\, dh\\
\leq \omega_Q\left(\frac{R}{\rho}\right)^Q\fint_{B(g_0,R)}u\, dh\leq \omega_Q\left(\frac{R}{\rho}\right)^Q\left(\fint_{B(g_0,R)}u^2\, dh\right)^{1/2}\\
=\frac {4^Q\omega_Q}{\left(1-\tau\right)^Q}\left(\fint_{B(g_0,R)}u^2\, dh\right)^{1/2}.
\end{multline}
where $\omega_Q$ is the volume of the unit gauge ball.

In order to estimate the integral in the tail  over $\bG\setminus B(g_0,R)$, we use the triangle inequality, $h\notin B(g_1,\rho)$ and $g_1\in B(g_0,\tau R)$, which give
\[
\frac {|h^{-1}\cdot g_0|}{|h^{-1}\cdot g_1|}\leq \frac {|h^{-1}\circ g_1|+|g_1^{-1}\cdot g_0|}{|h^{-1}\cdot g_1|}\leq 1+\frac {\tau R}{\rho}=1+\frac {4\tau}{1-\tau}=\frac {1+3\tau}{1-\tau}<\frac {4}{1-\tau}=\frac {R}{r}.
\]
Hence, we have
\begin{multline}\label{e:tail 2nd}
\rho^{2s}\int_{\bG\setminus B(g_0,R)} \frac{u(h)}{|h^{-1}\cdot g_1|^{Q+2s}}\, dh \leq \rho^{2s}\left(\frac {R}{\rho} \right)^{Q+2s}\int_{\bG\setminus B(g_0,R)} \frac{u(h)}{|h^{-1}\cdot g_0|^{Q+2s}}\, dh\\
=\left(\frac {R}{r} \right)^{Q} T(u;g_0,R)=\frac {4^Q}{\left(1-\tau \right)^{Q} } T(u;g_0,R).
\end{multline}
Inequalities \eqref{e:sup est tau 1}, \eqref{e:tail 1st} and \eqref{e:tail 2nd} give
\begin{multline}
M_{\tau R} \leq \frac {C_02^{Q/2}}{(1-\tau)^{Q/2}} \left( \fint_{B(g_0,R)} u^{2}\vol\right)^{1/2}+ \frac {{C_0}4^Q}{\left(1-\tau \right)^{Q} } T(u;g_0,R) + \frac {{C_0}4^Q\omega_Q}{\left(1-\tau\right)^Q}\left(\fint_{B(g_0,R)}u^2\, dh\right)^{1/2}\\
\leq \frac {(C_0+\omega_Q)4^Q}{(1-\tau)^{Q}}\left( \fint_{B(g_0,R)} u^{2}\vol\right)^{1/2}\ +\ \frac {C_04^Q}{\left(1-\tau \right)^{Q} } T(u;g_0,R)
\end{multline}
since $0<1-\tau<1$. The proof of \eqref{e:sup r and R} is complete.

Let us note that for $r$ and $R$ as in \eqref{e:sup r and R} satisfying, in addition, $ R_0/2\leq r<R\leq R_0$  we have ${R}/{(R-r)}\leq 2$ and
\[
 T(u;g_0,R)\leq 2^{2s}\left(\frac {R}{R_0} \right)^{2s}T(u;g_0,\frac {R_0}{2})\leq 2^{2s}T(u;g_0,\frac {R_0}{2}).
\]
Therefore, inequality \eqref{e:sup r and R} implies that for all $g_0\in \bG$ with $|g_0|=2R_0\geq 4\bar R_0$ and $ R_0/2\leq r<R\leq R_0$ we have
\begin{multline}\label{e:sup r and R 2}
M_r\leq 2^{Q+2s}C_1 \left[\left(\fint_{B(g_0,R)}
u^{2}  \vol \right)^{1/2} + T(u; g_0,\frac {R_0}{2})\right]\\
\leq M_R^{1/2}\, 2^{Q+2s}C_1 \left(\frac {1}{\omega_Q
R^Q}\int_{B(g_0,R_0)} u \vol \right)^{1/2} + 2^{Q+2s}C_1 T(u;
g_0,\frac {R_0}{2}).
\end{multline}

Inequality \eqref{e:sup r and R 2} implies, using $ab\leq \frac 12(a^2+b^2)$ and $R-r<R$, the inequality
\begin{equation}\label{e:sup r and R 3}
M_r\leq \frac 12 M_R + \frac {A}{(R-r)^{Q}}+B,
\end{equation}
where
\[
A=4^{Q+2s}C_1^2 \frac {1}{\omega_Q }\int_{B(g_0,R_0)}
u\vol \quad\text{and}\qquad B= 2^{Q+2s}C_1 T(u; g_0,\frac {R_0}{2}).
\]
Therefore, by a standard iteration argument, see for example
\cite[p. 191 Lemma 6.1]{Giusti03},
there exists a constant $c_Q$ so that $M_r\leq c_Q\left[
{A}{(R-r)^{-Q}}+B\right]$. Hence, for any $R_0\ge 2\bar R_0$, and
$g_0\in \bG$ with $|g_0|=2R_0$, we have
\[
\sup_{B(g_0,R_0/2)} \leq C \left[ \fint_{B(g_0,R_0)} u\vol \ +\  T(u; g_0,R_0/2)\right].
\]
This completes the proof   of Theorem \ref{t:loc bound}.

\section{Proof of Theorem \ref{t:asympt yamabe}}\label{S:main}

Recall that here we are considering a nonnegative subsolution $u$ to the Yamabe type equation $\frlap u=u^{2^*(s)-1}.$

\subsection{The optimal Lorentz space regularity}\label{ss:alt proof of Lorentz regul for Yamabe}
The first step is to obtain the optimal Lorentz space regularity of $u$. For this  we can adapt to the current setting \cite[Proposition 3.2 \& Proposition 3.3]{BMS16}, which give
\begin{equation}\label{e:sharp Lp regl}
u\in L^{r,\infty}(\bG)\cap L^\infty(\bG),
\end{equation}
recalling that $r= {2^*(s)}/{2}$, cf. \eqref{e:r def}.  Notice that in the cited results from \cite{BMS16}, valid in the Euclidean setting, the authors do not assume that the solution is radial, but the radial symmetry is used ultimately   to obtain the rate of decay of the solution of the fractional Yamabe equation.

For the sake of completeness  and self-containment of the proof, in the setting of a homogeneous group, and right-hand side of the equation modelled on the fractional Yamabe equation, we include a proof of the sharp Lebesgue space regularity \eqref{e:sharp Lp regl}, relying on Proposition \ref{p:Lpregul}. First, Proposition \ref{p:Lpregul}  implies that $u\in L^{q}(\bG)\cap L^\infty(\bG)$, for any $q>r=2^*(s)/2$.  Indeed, if $ V=u^{2^*(s)-2}$ then
since $u\in L^{2^*(s)}(\bG)$ it follows that $V\in L^{r'}(\bG)$. Hence, by Proposition \ref{p:Lpregul} b) it follows $u\in L^{q}(\bG)$ for all $q$ such that $\frac {2^*(s)}{2}< q<\infty$. Hence,  part c) gives  that we also have $u\in L^{\infty}(\bG)$.
Finally, we can see that $u \in L^{2^*(s)/2,\infty}(\bG)$ as follows. Take $F_t(u)=\min\{u,t\}$. Using the equation and the fractional Sobolev inequality we have
\begin{equation}\label{e:weak 1}
\norm{F_t\circ u}_{L^{2^*(s)}(\bG)}^2\leq C\int_{\bG}VuF_t(u)dg,
\end{equation}
where $V=u^{2^*(s)-2}$. Using first that $F_t(u)\leq t$ and then the definition of $V$  we have
\begin{equation}\label{e:decay of VuF}
\int_{\bG}VuF_t(u)dg\leq t\int_{\bG}Vu dg \leq t\int_{\bG}u^{2^*(s)-1}dg<\infty
\end{equation}
since $u^{2^*(s)-1}\in L^1(\bG)$ noting that  $2^*(s)-1>2^*(s)/2$.
Let $\mu(t)$ be the distribution function of $u$. From the definition of $F_t$ we have trivially
\begin{equation}\label{e:weak 3}
\int_{\bG}(F_t( u))^{2^*(s)}dg= t^{2^*(s)}\mu(t)+ \int_{\{u<t\}} u^{2^*(s)} dg \geq t^{2^*(s)}\mu(t).
\end{equation}
Therefore, bounding from above  the left-hand side of the above inequality using \eqref{e:weak 1} and then using \eqref{e:decay of VuF} we have
\[
t^{2^*(s)}\mu(t)\leq \left(C\int_{\bG}VuF_t(u)dg \right)^{2^*(s)/2}\leq C t^{2^*(s)/2},
\]
which shows that $u \in L^{2^*(s)/2,\infty}(\bG)$.

\subsection{Asymptotic behavior of the tail term}

We shall reduce the problem to a question of  $L^p$ regularity of certain truncated powers of the homogeneous norm, which we define next
For $R>0$ and $\alpha>0$ let
\[
\rho(g)=\rho_{\alpha,R}(g)=\left\{
                  \begin{array}{ll}
                    |g|^{-\alpha}, & |g|\ge R\\
                    0, & |g|<R.
                  \end{array}
                \right.
\]
\begin{lemma}\label{l:Lorentz regularity gauge cut-offs} For $Q/p<\alpha$ the Lorentz norms of $\rho_{\alpha,R}$ are given by the following formulas,
\begin{equation}\label{e:Lorentz norm of radial cut}
\norm{\rho_{\alpha,R}(g)}_{L^{p,\sigma}}\ =\  \left[ \int_0^\infty \left( t^{1/p}\rho^*(t)\right)^\sigma \frac {dt}{t}\right] ^{1/\sigma}\ = \  \frac{C_{Q,\sigma}}{R^{\alpha-\frac Q{p}}}.
\end{equation}
\end{lemma}

\begin{proof}
Let  $\mu(s)=\left\vert \left\{  g\mid \rho_{\alpha,R}(g)>s \right\}\right\vert$  be  the distribution function of $\rho_{\alpha,R}$. From the representation of the Haar measure in polar coordinates \eqref{e:polar coord}, we have
\[
\mu(s)  =\left\{
              \begin{array}{ll}
                  0, & s>\rho(R)\\
                  \frac{\sigma_Q}{Q} \left( s^{-Q/\alpha} - R^Q \right),& 0<s \le
                  \rho(R).
              \end{array}
         \right.
\]
 The corresponding radially decreasing rearrangement is
\[
\rho^*(t)=\inf \left\{ s\ge 0 \mid \mu(s)\le t\right\} = \left(
\frac {Q}{\sigma_Q}t +R^Q\right)^{-\alpha/Q}
\]
since $s=\rho^*(t)$ is determined from $\sigma_Q/Q
\left(s^{-Q/\alpha}-R^Q \right)=t>0$. A small calculation shows
then that  for some constant $C_{Q,\sigma}$ we have \eqref{e:Lorentz norm of radial cut}.
\end{proof}
Next, we use  the  optimal Lorentz space estimate  and the above Lemma \ref{l:Lorentz regularity gauge cut-offs}  to bound the tail.
\begin{lemma}\label{l:tail decay}
With the standing assumption, i.e.,  $u\in \mathcal{D}^{s,2}(\bG)$ is a nonnegative subsolution to the Yamabe type equation \eqref{e:frac yamabe}, we have that the tail has the following decay,
$$T(u;g_0,R)\equiv R^{2s}\int_{\bG\setminus B(g_0,R)} \frac {u(g)}{|g^{-1}\cdot g_0|^{Q+2s}}{dg}\ \leq C R^{-({Q-2s})}$$
with $C$ a constant depending on the homogeneous dimension $Q$.
\end{lemma}

\begin{proof}
By H\"older's inequality we have
\begin{equation}\label{e:Ho for tail}
T(u;g_0,R)\equiv R^{2s}\int_{\bG\setminus B(g_0,R)} \frac {u(g)}{|g^{-1}\cdot g_0|^{Q+2s}}{dg}\le R^{2s}\norm{u}_{L^{r,\infty}} \norm {\rho_{Q+2s,R}}_{L^{r',1}},
\end{equation}
recalling the definition of $r$ in  \eqref{e:r def} and using the weak $L^{r,\infty}$ regularity of $u$ that we already proved. Hence, the claim of the Lemma follows by Lemma \ref{l:Lorentz regularity gauge cut-offs} which  shows that for some constant $C=C(Q)$ we have
\begin{equation}\label{e:gauge norm decay}
\norm {\rho_{Q+2s,R}}_{L^{r',1}}\leq CR^{-Q}.
\end{equation}
As a consequence, taking into account that $r'=Q/(2s)$ we obtain \eqref{e:gauge norm decay}.

\end{proof}

\subsection{The slow decay}
The  proof of Theorem \ref{t:asympt yamabe} will also use a preliminary ''slow''
 decay of the solution $u$, see \cite[Lemma 2.1]{Z} for case of the Yamabe equation on a Riemannian manifold with maximal volume growth.

\begin{lemma}\label{l:slow decay}
If  $u\in \mathcal{D}^{s,2}(\bG)$ is a nonnegative subsolution to the
Yamabe type equation, then $u$ has the slow decay
$|g|^{({Q-2s})/2}u\in L^\infty(\bG)$.
\end{lemma}

\begin{proof}
The key to this decay is  the scale invariance of the equation, i.e.,  the fact that $$u_\lambda(g)=\lambda^{({Q-2s})/2}u(\delta_\lambda g)$$ is also a subsolution to  the Yamabe type equation and  the scale invariance of the $\mathcal{D}^{s,2}(\bG)$ and the $L^{2^*(s)}(\bG)$ norms. %\norm{u}_{\mathcal{D}^{s,2}
In order to show the slow decay, it is then enough to show that
there exist constants $\lambda_o$ and    $C$, depending only on
$Q$ and $s$, and the invariant under the scaling norms, such that for all
$g_0$ with  $\lambda=|g_0|/2>\lambda_0$   we have on the ball
$B(h_0,1)$ with $h_0= \delta_{\lambda^{-1}}g_0$, the estimate
\begin{equation}\label{e:bound for v}
\max_{h\in B(h_o,1)}u_\lambda(h) \leq C.
\end{equation}
Indeed, \eqref{e:bound for v} implies
\[
\left(\frac {|g_o|}{2}\right)^{(Q-2s)/2}u(g_o)\leq \left(\frac
{|g_o|}{2}\right)^{(Q-2s)/2}\sup_{B(g_0,\lambda) } u(
g)=\max_{B(h_0,1)}u_{\lambda}(h)\leq C,
\]
which gives the desired decay. The  bound \eqref{e:bound for v} will be seen  from the local version of Proposition \ref{p:Lpregul} c)  in the case $V=u_\lambda^{2^*(s)-2}$ by showing  that the local supremum bound is independent of $\lambda$. To simplify the notation let $v=u_\lambda $. We follow the argument in the proof of Theorem \ref{t:loc bound} with $V=v^{2^*(s)-2}$. Furthermore, for $\frac 12<r<R<\frac 32$ we take a bump function $\psi$, so that,
\[
\psi\vert_{B_r}\equiv 1,\qquad \mathrm{supp}\, \psi\Subset B_{\frac {r+R}{2}},
\]
where here and for the remainder of the proof, for any $r>0$  we will denote by $B_r$ the ball $B(h_0,r)$ with the understanding that the center is $h_0$.

In particular, we have \eqref{e:cacciopolli2} with $u_\delta$ replaced by $v_\delta$, but now we can absorb the first term  in the righthand side in the left hand side for all sufficiently large $\lambda$. Indeed, applying H\"older's inequality we have
\begin{equation}
\int\psi^2Vv_\delta^{\beta+1}\vol\leq \left[\int_{B_{\frac
{r+R}{2}}}V^{r'}\vol\right]^{1/r'} \left[ \int_{B_{\frac
{R+r}{2}}}\psi^{2^*(s)}v_\delta^{(\beta+1)2^*(s)/2 }\vol
\right]^{2/2^*(s)}.
\end{equation}
Since $V=v^{2^*(s)-2}$ the first term can be estimated as follows,
\[
\int_{B_{\frac {r+R}{2}}}V^{r'}dh=\int_{B_{\frac {r+R}{2}}}v^{2^*(s)}dh\leq \int_{B(g_0,\lambda)}u^{2^*(s)}\vol\rightarrow 0\quad\text{as}\quad \lambda\rightarrow \infty,
\]
using the scaling property  of the $L^{2^*(s)}$ norm and $u\in L^{2^*(s)}(\bG)$. Therefore, we have the analog of \eqref{e:ineq to start Moser}, i.e., for all $\lambda\ge  \lambda_0$ there exists a constant $C= C(Q,s, K_0)$, such that, the following inequality holds true
\begin{equation}\label{e:ineq to start Moser sale}
\left[ \int_{B_{r}}v_\delta^{(\beta+1)\frac {2^*(s)}{2} }
\vol\right]^{\frac {2}{2^*(s)}} \le \frac
{C\beta^{\kappa}}{R^{2s}}\left[\left(\frac {R}{R-r}
\right)^2\right]\left[1 +\frac {T(v;h_0,R)}{\delta } \right]
\int_{B_R}v_\delta^{\beta+1}\vol.
\end{equation}
A  Moser type iteration argument shows then the existence of a constant  $C$ such that for all $\lambda\ge  R_0$,  $h_0= \delta_{\lambda^{-1}}g_0$ and $|g|=2\lambda$ we have the inequality
\begin{multline}\label{e:v is locally unif bdd}
\sup_{B(h_0,1)} v\ \leq \ C \left[ \left (\fint_{B(h_0,2)} v^{2^*(s)} \vol \right)^{1/{2^*(s)}} + T(v; h_0,1/2)\right]\\
\leq C \left[ \left (\int_{\bG} v^{2^*(s)} \vol\right)^{1/{2^*(s)}} + T(v; h_0,1/2)\right]\\
\leq C\left[\norm{u}_{\DsG}+\norm{u}_{L^{2^*(s)}(\bG)} \norm {\rho_{Q+2s,1}}_{L^{2^*(s)',1}}\right]\le C,
\end{multline}
after using the fractional Sobolev inequality, H\"older's inequality,  \eqref{e:Lorentz norm of radial cut} and the invariance under scalings of the $\DsG$ and $L^{2^*(s)}(\bG)$ norms.

\end{proof}

\subsection{Conclusion of the  proof of Theorem \ref{t:asympt yamabe}}\label{ss:end of proof}
We begin by noting that, from what we have already proved,  Theorem \ref{t:loc bound} can be applied to the potential $V=u^{2^*(s)-2}$. Indeed, the slow decay of $u$, cf. Lemma \ref{l:slow decay}, gives that for some constant $C$ we have
\[
u(g)\leq C|g|^{-({Q-2s})/2},
\]
 which together with \eqref{e:polar coord} implies  the needed assumptions on $V$, in particular, for  $t_0>r'= {Q}/{(2s)}$, cf. \eqref{e:r def},  we have
\[
 \int_{|g|\geq
R}V^{t_0}\vol =\int_{|g|\geq
R} u^{t_0(2^*(s)-2)}\vol  \leq C^{\frac {4st_0}{Q-2s}}\int_{|g|\geq
R} |g|^{-2st_0}\vol
= \frac {\sigma_Q C^{\frac {4st_0}{Q-2s}}}{2st_0-Q}  \frac
{1}{R^{2st_0-Q
}}.
\]
Therefore, Theorem \ref{t:loc bound} gives that for all  $g\in \bG$ and $2R=|g|$ sufficiently large we have \eqref{e:loc bound}, i.e., there exists a constant $C$ independent of $g$, such that,
\begin{equation}\label{e:loc bound3}
\sup_{B(g,R/2)} u\ \leq \ C  \fint_{B(g,R)} u \ +\ {C}T(u; g,R/2).
\end{equation}

Furthermore, the weak $L^{2^*(s)/2}$ regularity \eqref{e:sharp Lp regl} shows that  for  $r=2^*(s)/2$ we have the inequality %for $1\leq q<q_0=2^*(s)/2$
\begin{equation}\label{e:average decaay from Lorentz}
\fint_{B(g,R)} u \vol\leq \frac {r}{r-1}\frac {1}{|B_R|^{1/{r}}}\norm{u}_{L^{2^*(s)/2,\infty}}=\frac {C}{R^{Q-2s}}\norm{u}_{L^{r,\infty}},
\end{equation}
taking into account that for $1\leq p<\infty$ the $L^{p,1}(\bG)$ norm of the characteristic function of the gauge ball $B_R$ is $p|B_R|^{1/p}$.

Now we are ready to conclude the proof of Theorem \ref{t:asympt yamabe} since by \eqref{e:loc bound3}, \eqref{e:average decaay from Lorentz} %\eqref{e:loc bound}
and Lemma \ref{l:tail decay} we can claim the following estimate  for all sufficiently large $2R=|g|$,
\[
u(g)\leq \max_{B(g,R/2)} u\ \leq \ \frac {C}{R^{Q-2s}}\norm{u}_{L^{2^*(s)/2,\infty}}\ +\ \frac {C}{R^{Q-2s}}
\]
with a constant $C$ independent of $g$.

%\nocite{*}

\end{document}